\documentclass[12pt, a4paper]{amsart} 

\usepackage{geometry}
\usepackage{amsmath, amssymb, amsfonts, amsthm}
\usepackage{mathabx}

\newcommand{\BMO}{\mathrm{BMO}}
\newcommand{\R}{\mathbb{R}}
\newcommand{\Z}{\mathbb{Z}}
\newcommand{\F}{\mathcal{F}}
\newcommand{\wh}{\widehat}
\newcommand{\wt}{\widetilde}
\newtheorem{defin}{Definition}

\DeclareMathOperator{\supp}{supp}

\newtheorem{theorem}{Theorem}

\newtheorem{theom}{Theorem}

\newtheorem{lem}{Lemma}

\title[Littlewood--Paley characterization of $\BMO$ and Triebel--Lizorkin]{Littlewood--Paley characterizations of $\BMO$ and Triebel--Lizorkin spaces}

\author{Anton Tselishchev}
\thanks{This research was supported by the Russian Science Foundation (grant No.~18-11-00053).}
\address{Chebyshev Laboratory, St. Petersburg State University, 14th Line V.O., 29B, Saint Petersburg 199178 Russia}
\address{St. Petersburg Department of Steklov Mathematical Institute, Russian Academy of Sciences (PDMI RAS), Russia}
\email{celis-anton@yandex.ru}

\author{Ioann Vasilyev}
\address{Universit\'e Paris-Est, LAMA (UMR 8050), 5 BOULEVARD DESCARTES, 77454, CHAMPS SUR MARNE, France}
\address{St. Petersburg Department of Steklov Mathematical Institute, Russian Academy of Sciences (PDMI RAS), Russia}
\email{milavas@mail.ru}


\date{}

\begin{document}

\begin{abstract}
We prove one generalization of the Littlewood--Paley characterization of the $\BMO$ space where the dilations of a Schwartz function are replaced by a family of functions with suitable conditions imposed on them. We also prove that a certain family of Triebel--Lizorkin spaces can be characterized in a similar way.
\end{abstract}

\maketitle

\section{Introduction}
One of the main results of this article is a generalization of the following statement, proved in the paper~\cite{tselvasl}.

\begin{theom}
\label{thm1}
 Let $\{\psi_n\}_{n\in\Z}$ be a uniformly bounded system of functions defined on $\R^d$, having weak derivatives up to the order $d+1.$ Suppose that the following conditions hold.\\
 \emph{1)}\, $\sum_{n\in\Z}\psi_n(x) \equiv 1$ for all $x\neq 0$.\\
 \emph{2)}\, $\supp\psi_n\subseteq\{x\in\R^d: 2^{n-1}\leq |x| < 2^{n+1}\}$.\\
 \emph{3)}\, $2^{-nd}\int |D^{\alpha}\psi_n(\xi)| d\xi \leq K 2^{-n|\alpha|}$ for $0\leq |\alpha|\leq d+1$.\\
 Here $K$ is a constant which does not depend on $n$.
 Define an operator $\widehat{\Delta_n f}:=\psi_n\hat{f}$ and a norm
 $$
 \|f\|_D:=\sup_{Q}\Big(\frac{1}{|Q|}\int_Q \sum_{2^{-n}\leq l(Q)}|\Delta_n f(x)|^2 dx\Big)^{1/2},
 $$
 where we take the supremum over all cubes $Q$, and $l(Q)$ is the length of the edge of $Q$. Then $C_1\|f\|_D \leq \|f\|_{\BMO}\leq C_2\|f\|_D$ 
 for some positive constants $C_1$ и $C_2$.
\end{theom}
This theorem is in turn a generalization of Bochkarev's inequality (see~\cite{Bochkarev1}) obtained via replacing the de la Vall\'ee--Poussin kernels used by Bochkarev with a more general system of functions. We draw the reader's attention to the fact that this inequality was applied by S.V.Bochkarev to various problems of the theory of trigonometric sums, see~\cite{Bochkarev2} for further details.

In this paper we prove a generalization of Theorem~\ref{thm1} obtained by substituting $\{\psi_n\}_{n\in \mathbb Z}$ with an even more general system of functions. In more details we completely dispense with the ``$d+1$ derivatives condition'', demanding bounds on the derivatives up to the order $[d/2]+1$ instead. The price to pay is the norm utilized, which we consider to be $L^2$ and not $L^1$. Let us state the main result of this paper.

\begin{theorem}
\label{thm2}
 Let $\{\psi_n\}_{n\in\Z}$ be a uniformly bounded system of functions defined on $\R^d$, having weak derivatives up to the order $a=[d/2]+1.$ Suppose that the following conditions hold.\\
 \emph{1)}\, $\sum_{n\in\Z}\psi_n(x) \equiv 1$ for all $x\neq 0$.\\
 \emph{2)}\, $\supp\psi_n\subseteq\{x\in\R^d: 2^{n-1}\leq |x| < 2^{n+1}\}$.\\
 \emph{3)}\, ($2^{-nd}\int |D^{\alpha}\psi_n(\xi)|^2 d\xi)^{1/2} \leq K 2^{-n|\alpha|}$ for $0\leq |\alpha|\leq a$.\\
 Here $K$ is a constant which does not depend on $n$.
 Define an operator $\widehat{\Delta_n f}:=\psi_n\hat{f}$ and a norm
 $$
 \|f\|_D:=\sup_{Q}\Big(\frac{1}{|Q|}\int_Q \sum_{2^{-n}\leq l(Q)}|\Delta_n f(x)|^2 dx\Big)^{1/2},
 $$
 where we take the supremum over all cubes $Q$, and $l(Q)$ is the length of the edge of $Q$. Then $C_1\|f\|_D \leq \|f\|_{\BMO}\leq C_2\|f\|_D$ 
 for some positive constants $C_1$ и $C_2$.
\end{theorem}
\noindent
Several remarks are in order. First, Theorem~\ref{thm2} is indeed a generalization of Theorem~\ref{thm1}; this can be easily proved by the Sobolev--Gagliardo--Nirenberg inequality. Second, we emphasize that the conditions imposed in Theorem~\ref{thm2} are exactly those of the H\"ormander--Mikhlin multiplier theorem (contact~\cite{ClassGrafakos} or~\cite{RdF} for the proof). Finally, we remark that it would not be difficult to obtain a variant of the Littlewood--Paley decomposition for the product space $\BMO(\R^{d_1}\times\R^{d_2})$ in the spirit of Theorem~\ref{thm2} along the lines of this result. However we do not attack this problem here. 

Theorem~\ref{thm2} provides a Littlewood--Paley characterization of the space $\BMO(\R^d).$ In the second main result of this article we establish such a characterization for the scale of Triebel--Lizorkin spaces. We introduce those in the following

\begin{defin} Let $\varphi$ be a collection of functions on $\mathbb R^d,$  $\varphi= \{ \varphi_n\}_{n=1}^{\infty},$ such that 
\begin{enumerate}
	\item [1)] \label{one} $\mathrm{supp} \; \varphi_n \subseteq \{x \in \mathbb R^d :2^{n-1} \leq |x|<2^{n+1} \} $,
   \item [2)] \label{two} $\sum_{n \in \mathbb Z} \varphi_n \left( x \right) = 1 \text{ for all}\; x \neq 0$,
 \item [3)] \label{three} $2^{-nd} \int_{\mathbb R^d} |D^\alpha \varphi_n \left( \xi \right)|d\xi \leq K_\varphi \cdot 2^{-n|\alpha|} \; \text{for all}\; 0 \leq|\alpha | \leq d+1.$
				\end{enumerate}
The Triebel--Lizorkin space $\mathfrak{F}_{\infty, \varphi}^{0,p},$ for $p \in \left( 1, + \infty \right)$ is defined as follows: 
$f \in\mathfrak{F}_{\infty, \varphi}^{0,p} $ if and only if 
$$\|f\|_\varphi ^p := \sup \limits_Q \left(\frac{1}{|Q|} \int \limits_Q \sum \limits_{2^{-n} \leq l\left( Q \right)} |\Delta_{n,\varphi} f\left( x \right)|^p dx\right)< \infty ,$$ where the $\sup$ is taken over all cubes $Q \subset \mathbb R^d$, $l \left(Q\right)$ is the length of an edge of $Q,$ and operators $\Delta_{n,\varphi}$ are defined as:
$$\widehat{\Delta_{n,\varphi} f}\left(x\right) = \varphi_n \left(x\right) \cdot \widehat{f\left(x\right)}.$$
\end{defin}
\noindent
The classical Triebel--Lizorkin spaces arise in a particular case of this definition, namely when the functions $\varphi_n$ are dilatations of some function from the Schwartz class, see~\cite{ModGrafakos}. Our second main result here is the following theorem.
\begin{theorem}
\label{tribl1}
Let  $\varphi = \{\varphi_n \}_{n=1}^{\infty}$ and $\psi =\{\psi_n \}_{n=1}^{\infty}$ be two collections of functions satisfying the conditions of the definition. Then
$$ \|f\|_\varphi \lesssim \|f\|_\psi,$$ where the constant depends only on $p$ and $K_\psi$
\end{theorem}
\noindent
Despite the fact that it would be easy to strengthen Theorem~\ref{tribl1} via imposing a condition demanding only $[d/2]+1$ derivatives as we made it in Theorem~\ref{thm2}, we intentionally keep the stronger condition with $d+1$ derivatives (the one used in the paper~\cite{tselvasl}) and the norm $L^1$ here in order for the reader to compare the proofs of Theorems~\ref{tribl1} and~\ref{thm2}.

Let us also point out that obtaining a Littlewood--Paley decomposition for the scale of Triebel--Lizorkin spaces is a more difficult problem in comparison with the scale of Besov spaces (see~\cite{ModGrafakos}, pages 73--76). One more remark is that it would be interesting (at least in our opinion) to compare the results of the present paper with those of~\cite{finiki}, since similar characterizations were used there in order to describe some important properties of Triebel--Lizorkin and Besov spaces.

The authors are kindly grateful to their scientific adviser Sergei V. Kislyakov for having posed the problem and for the continuous support during the process of its solution.

We begin with the characterization of the scale of Triebel--Lizorkin spaces.
\section{The Triebel--Lizorkin spaces}
\begin{proof}
Denote $S_{n,\varphi}:=\check {\varphi_n}$ and $P_n :=S_{n-1} +S_n +S_{n+1}.$ Then
$$\Delta_{n,\varphi}f\left( x \right) = S_{n,\varphi} \ast f\left( x \right) = S_{n,\varphi} \ast P_{n,\psi} \ast f\left( x \right).$$
Indeed this follows from the facts that $\sum_n \psi_n =1$ and that $\mathrm{supp }\; \psi_n \subseteq \{ x \in \mathbb R^d : 2^{n-1} \leq |x| <2^{n+1} \}.$
We fix a cube $Q \subset \mathbb R^d$ and write
\begin{equation}
\begin{split}
\label{pyatodin}
&\frac{1}{|Q|} \int \limits_Q \sum \limits_{n \geq - \log_2 l\left( Q \right)} |\Delta_{n,\varphi} f \left(x\right) |^p dx=\\
&\frac{1}{|Q|} \int \limits_Q  \sum \limits_{n \geq - \log_2 l\left( Q \right)} | f \ast S_{n,\varphi} \ast P_{n,\psi} \left( x \right) |^p dx\leq\\
&\frac{1}{|Q|} \int \limits_Q  \sum \limits_{n \geq - \log_2 l\left( Q \right)} \biggl| \int \limits_{2Q} f  \ast P_{n,\psi} \left( y \right) \cdot S_{n,\varphi}  \left(x - y \right) dy  \biggr|^p dx +\\
&\frac{1}{|Q|} \int \limits_Q  \sum \limits_{n \geq - \log_2 l\left( Q \right)}  \biggl| \int \limits_{{\mathbb R^d}\backslash 2Q} f  \ast P_{n,\psi} \left( y \right)\cdot S_{n,\varphi}  \left(x - y \right) dy   \biggr|^p dx ,
\end{split}
\end{equation}
where $2Q$ stands for the cube whose center coincides with that of $Q$, and whose edge is two times longer than the
edge of $Q.$

We first estimate the first integral:
\begin{equation}
\begin{split}
&I := \frac{1}{|Q|} \int \limits_Q  \sum \limits_{n \geq - \log_2 l\left( Q \right)}  \biggl| \int \limits_{2Q} f  \ast P_{n,\psi} \left( y \right) \cdot S_{n,\varphi}  \left(x - y \right) dy   \biggr|^p dx \lesssim\\
&\frac{1}{|Q|} \sum \limits_{n \geq - \log_2 l\left( Q \right)} \| \Delta_{n,\varphi} \left( \left( f  \ast P_{n,\psi} \right) \cdot \chi_{2Q} \right) \|_{L^p \left( \mathbb R^d \right)}^p .
\end{split}
\end{equation}
Using the Young inequality we conclude that for all $g \in L^p \left( \mathbb R^d \right)$ the following inequality holds:
$$ \| \Delta_{n,\varphi}\: g\|_{L^p \left( \mathbb R^d \right)} = \|S_{n,\varphi} \ast g \|_{L^p \left( \mathbb R^d \right)} \leq \| S_{n,\varphi} \|_{L^1 \left( \mathbb R^d \right)} \cdot \|g\|_{L^p \left( \mathbb R^d \right)}.$$

Next, for all $x \in \mathbb R^d$ we have the inequalities $$|S_{n,\varphi} \left( x \right)| \lesssim 2^{nd}$$ and $$|S_{n,\varphi} \left( x \right)| \lesssim 2^{-n} \cdot |x|^{- \left( d + 1 \right)}.$$
Indeed the first estimate is a piece of cake: 
$$|S_{n, \varphi}\left( x \right)| =  \biggl| \int \limits_{\mathbb R^d} \varphi_n \left( \xi \right) \cdot e^{2\pi i x\cdot \xi} d\xi  \biggr| \leq \| \varphi_n \|_{L^1 \left( \mathbb R^d \right)} \lesssim 2^{nd}.$$
In order to prove the second one we first infer that it is enough show that 
$$ | x_j ^{d+1} \cdot S_{n, \varphi} \left( x \right)| \lesssim 2^{-n}$$
for all $j \in [1,\ldots,d].$ We are going to prove these in the following way: we majorize the left--hand side of the last inequality using the properties of the Fourier transformation and the condition 3 with the corresponding multi--index $d_j:$
\begin{equation*}
\begin{split}
 &| x_j ^{d+1} \cdot S_{n, \varphi} \left( x \right) | = | x_j ^{d+1} \cdot \check{\varphi}_n | \lesssim | \widecheck {D^{d_j}\varphi_n}\left( x \right)| \lesssim \\
 &\| D^{d_j} \varphi_n \left( x \right)\|_{L^1 \left( \mathbb R^d \right)} \lesssim 2^{nd - n|d_j |} = 2^{nd-n\left( d+1 \right)} = 2^{-n},
\end{split}
\end{equation*} 
and the second estimate follows as well. 

Note that these two inequalities yield that
\begin{equation*}
\begin{split}
&\| S_{n, \varphi} \|_{L^1 \left( \mathbb R^d \right)}  \leq \int \limits_{\{x:|x|\leq 2^{-n}  \}} | S_{n, \varphi} \left( x \right) | dx + \int \limits_{\{x:|x| > 2^{-n}  \}}  | S_{n, \varphi} \left( x \right) | dx \lesssim\\
&2^{nd} \cdot 2^{-nd} + \int \limits_{\{x:|x| > 2^{-n}\}} 2^{-n}\cdot |x|^{-\left(d+1\right)} dx \lesssim 1.
\end{split} 
\end{equation*}
Let us now continue the estimate of the term $I$:
\begin{equation*}
\begin{split}
&I\lesssim \frac{1}{|Q|} \sum \limits_{n \geq-\log_2 l \left( 2Q \right) } \int \limits_{2Q} |f \ast P_{n, \psi} \left( \xi \right)|^p d\xi \lesssim\\
&\frac{1}{|2Q|} \sum \limits_{n \geq-\log_2 l \left( 2Q \right) } \int \limits_{2Q} |f \ast S_{n, \psi} \left( \xi \right)|^p d\xi \lesssim  \|f\|^p_\psi.
\end{split}
\end{equation*}
We proceed to the second term from~\eqref{pyatodin}, which we denote $J$:
$$ J := \frac{1}{|Q|} \int \limits_Q  \sum \limits_{n \geq - \log_2 l\left( Q \right)}  \biggl| \int \limits_{{\mathbb R^d}\backslash 2Q} f  \ast P_{n,\psi} \left( y \right) \cdot S_{n,\varphi}  \left(x - y \right) dy \biggr|^p dx.$$ 
Let us first estimate the expression that is inside of the integral over the set $Q$:
$$ \biggl| \int \limits_{{\mathbb R^d}\backslash 2Q} | f \ast P_{n,\psi} \left( y \right) \cdot S_{n,\varphi}  \left(x - y \right) dy \biggr| \leq \sum \limits_{i=2} ^\infty \int \limits_{\Omega_i}| f \ast P_{n,\psi} \left( y \right)| \cdot |S_{n,\varphi}  \left(x - y \right)| dy \leq \ldots\, ,$$
where $\Omega_i = \left(i + 1 \right) Q \backslash iQ. $ We use the bound $ |S_n \left(x \right) | \leq 2^{-n} \cdot |x|^{-\left(d + 1 \right)}$ and the fact that if $x\in Q$ and $y\in \Omega_i$, then $|x-y|\asymp i l(Q)$ and write
\begin{equation*}
\begin{split}
&\ldots \leq \sum \limits_{i=2} ^\infty \int \limits_{\Omega_i} \frac{2^{-n} \cdot |f \ast P_{n, \psi} \left( y \right)|} {\left( i \cdot l\left(Q \right) \right)^{d+1}}dy \leq \frac{2^{-n}}{l \left(Q \right)^{d+1}} \cdot \sum \limits_{i=2}^\infty \frac{\|f\|_\psi \cdot |\Omega_i|}{i^{d+1}} \leq\\
&\frac{2^{-n} \cdot l\left(Q\right)^d}{l\left(Q\right)^{d+1}}\cdot \|f\|_\psi \cdot \sum \limits_{i=2}^\infty \frac{1}{i^2} \lesssim \frac{2^{-n}}{l\left(Q\right)} \cdot \|f\|_\psi,
\end{split}
\end{equation*}
where in the last inequality we have used the lemma that follows.
\begin{lem}
Let $\Omega_i$ be as above. Then
$$\frac{1}{|\Omega_i|} \int \limits_{\Omega_i} |f \ast P_{n,\psi}\left(y\right)| dy \lesssim \|f\|_\psi, $$
and the lemma follows.
\end{lem}
\begin{proof}
Chop up the set $\Omega_i $ into pairwise disjoint cubes $ \{ Q_j \}_{j=1} ^N ,$ in a way that each of those is an image of $Q$ under a translation. Then the number $N$ of these cubes satisfies $N \lesssim i^{d-1}$. Note that for each $j \in[1,\ldots,N],$ 
\begin{equation*}
\begin{split}
&\biggl(\frac{1}{|Q_j|} \int \limits_{Q_j} |f \ast P_{n, \psi}\left(y\right)|^p dy\biggr)^{\frac{1}{p}} \lesssim\\
&\biggl(\frac{1}{|Q_j|} \int \limits_{Q_j} |f \ast S_{n-1, \psi} \left(y\right)|^p + |f \ast S_{n, \psi} \left( y \right) |^p + |f \ast S_{n+1, \psi}\left( y \right)|^p dy\biggr)^{\frac{1}{p}} \lesssim \|f\|_\psi.
\end{split}
\end{equation*}
Using this bound together with the H\"older inequality we infer that
\begin{equation*}
\begin{split}
&\int \limits_{\Omega_i} |f \ast P_{n, \psi} \left( y \right) |dy = \sum \limits _{j=1}^N \int \limits_{Q_j} |f \ast P_{n, \psi} \left(y \right) |dy \leq \\
&\sum \limits_{j=1}^N \biggl(\frac{1}{|Q_j|} \int \limits_{Q_j} | f \ast P_{n, \psi} \left( y \right)|^p dy\biggr)^{\frac{1}{p}} \cdot |Q_j| \leq \|f\|_\psi \cdot N \cdot |Q_j| \lesssim \\
&\|f\|_\psi \cdot i^{d-1} \cdot |Q| \lesssim \|f\|_\psi \cdot |\Omega_i|.
\end{split}
\end{equation*}
\end{proof}
We are ready now to finish off the estimate of the term $J$:
$$ J \lesssim \frac{1}{|Q|}\int \limits_Q \sum \limits_{n\geq - \log_2 l \left(Q \right)} \frac{\|f\|_\psi ^p 2^{-np}}{l \left( Q \right)^p}dx \lesssim\|f\|_{\psi}^p \cdot \frac{l \left(Q\right)^p}{l \left(Q\right)^p} \lesssim \|f\|_\psi^p ,$$
and the theorem follows.
\end{proof}

\section{The $\BMO$ inequality}
\begin{proof}[Proof of theorem A]
 We write $f_Q$ for $\frac{1}{|Q|}\int_Q f(x)dx$. We consider the following norm in the space $\BMO$: 
$\sup_Q (\frac{1}{|Q|}\int_Q |f(x)-f_Q|^2 dx)^{1/2}$. We need to show that $\|f\|_{\BMO}\lesssim \|f\|_D$ 
 and $\|f\|_D\lesssim \|f\|_{\BMO}$.
\subsection{The inequality $\|f\|_D\lesssim\|f\|_{\BMO}$}
We begin with first part of the proof of the theorem. First of all we fix a cube $Q$. We are going to estimate the integral
$$
\Big(\frac{1}{|Q|}\int_Q \sum_{2^{-n}\leq l(Q)}|\Delta_n f(x)|^2 dx\Big)^{1/2}.
$$
We decompose $f$ into a sum of there functions. In more details, we write
$(f-f_Q)\chi_{2Q}+ (f-f_Q)\chi_{\R^d \setminus 2Q} +f_Q=:f_1+f_2+f_3$, where $\chi_A$ is the characteristic function of a set 
$A$. Note that $f_3$ is a constant, which means that 
$\Delta_n f_3=0$. Hence we infer that 
$$
\|f\|_D^2\lesssim\Big(\frac{1}{|Q|}\int_Q \sum_{2^{-n}\leq l(Q)}|\Delta_n f_1(x)|^2 dx\Big)^{1/2}+
 \Big(\frac{1}{|Q|}\int_Q \sum_{2^{-n}\leq l(Q)}|\Delta_n f_2(x)|^2 dx\Big)^{1/2}.
$$
The first term here is a piece of cake. Indeed, using the fact that $\{\psi_n\}$ are uniformly bounded, we can conclude that
$\|Sg\|_{L^2}\lesssim \|g\|_{L^2}$ where $Sg=(\sum_{n\in\Z} |\Delta_n g|^2)^{1/2}$. So, the first term 
 is less than or equal to
\begin{align*}
&\Big(\frac{1}{|Q|} \int_{\R^d} \sum_{n\in \Z} |\Delta_n((f-f_Q)\chi_{2Q})|^2 dx\Big)^{1/2}\lesssim\\
&\Big(\frac{1}{|Q|}\int_{2Q} |f(x)-f_Q|^2 dx\Big)^{1/2} \lesssim \Big( \frac{1}{|2Q|}\int_{2Q} |f(x)-f_{2Q}|^2 dx\Big)^{1/2}
+|f_{2Q}-f_Q|.
\end{align*}
Both expressions here are less than $C\|f\|_{\BMO}$ for some universal constant $C$. For the first term it is a consequence of the definition of the space $\BMO$, and for the second one one has to sum up the inequalities $\int_Q |f(x)-f_Q| dx \leq |Q|\|f\|_{\BMO}$ and 
$\int_Q |f(x)-f_{2Q}|dx\leq \int_{2Q} |f(x)-f_{2Q}| dx \lesssim |Q|\|f\|_{\BMO}$. Hence it is left to estimate
\begin{equation}
\frac{1}{|Q|} \int_Q \sum_{2^{-n}\leq l(Q)} |\Delta_n ((f-f_Q)\chi_{\R^d \setminus 2Q})(x)|^2 dx. \label{nineth}
\end{equation}
Let $x\in Q$. Let us rewrite the expression that is inside of the integral in~\eqref{nineth} (recall that $S_n=\check{\psi}_n$):
\begin{align}
&|\Delta_n ((f-f_Q)\chi_{\R^d \setminus 2Q})(x)|=\Big| \int_{\R^d \setminus 2Q} (f(y)-f_Q) S_n(x-y) dy\Big|\leq \nonumber \\
&\sum_{k=2}^{\infty} \int_{2^k Q\setminus 2^{k-1}Q}|f(y)-f_Q| |S_n(x-y)|dy\leq \nonumber \\
\label{anton1}
&\sum_{k=2}^{\infty} \left[\biggl(\int_{2^k Q\setminus 2^{k-1}Q}|f(y)-f_Q|^2\biggr)^{\frac{1}{2}}\biggl(\int_{2^k Q\setminus 2^{k-1}Q}|S_n(x-y)|^2dy\biggr)^{\frac{1}{2}}\right].
\end{align}
Next, we infer that
\begin{align*}
\label{vanya2}
 &\biggl(\int_{2^k Q\setminus 2^{k-1}Q} |f(y)-f_Q|^2\biggr)^{1/2} dy\leq
 \biggl(\int_{2^k Q} |f(y)-f_Q|^2\biggr)^{1/2} dy\leq\\
 &\biggl(\int_{2^k Q} |f(y)-f_{2^kQ}|^2\biggr)^{1/2} dy +|2^kQ|^{1/2}\cdot|f_{2^kQ}-f_Q|.
\end{align*}
The first term here does not exceed $|2^kQ|^{1/2}\|f\|_{\BMO}$ whereas the second one is less than or equal to $k|2^kQ|^{1/2}\|f\|_{\BMO},$ which can be easily proved by induction in $k.$ Let us estimate the second factor in~\eqref{anton1}. 
Note that once $x\in Q$ and $y\in 2^k Q\setminus 2^{k-1}Q,$ one has $|x-y|\asymp 2^kl(Q).$ Thanks to this fact we write
\begin{equation*}
\begin{split}
&\int_{2^k Q\setminus 2^{k-1}Q} |S_n(x-y)|^2 dy \lesssim
 \int_{2^k Q\setminus 2^{k-1}Q} \frac{|x-y|^{2a}}{2^{2ak}l(Q)^{2a}}|S_n(x-y)|^2 dy \lesssim \\
&\frac{1}{2^{2ak}l(Q)^{2a}} \int_{\R^d} |x|^{2a}|S_n(x)|^2 dx\lesssim \ldots\;.
\end{split}
\end{equation*}
We continue our estimates, now using the Plancherel theorem and condition 3 with multi--indexes $\alpha$ satisfying $|\alpha|=a$
\begin{equation*}
 \ldots\lesssim\frac{1}{2^{2ak}l(Q)^{2a}} 2^{nd-2na}.
\end{equation*}
The inequalities above yield
\begin{align*}
&|\Delta_n ((f-f_Q)\chi_{\R^d \setminus 2Q})(x)|\lesssim \sum\limits_{k=2}^{\infty} k(2^{kd}|Q|)^{1/2}\cdot 2^{\frac{nd}{2}-na} 2^{-ak} l(Q)^{-a} \|f\|_{\BMO}= \\ 
&l(Q)^{d/2-a}2^{n(d/2-a)}\|f\|_{\BMO}\sum_{k=2}^\infty k\cdot 2^{k(d/2-a)}\lesssim l(Q)^{d/2-a}2^{n(d/2-a)}\|f\|_{\BMO}.
\end{align*}
Finally we are able to write the following estimate for the expression~\eqref{nineth}:
\begin{align*}
&\biggl(\frac{1}{|Q|}\int_Q\sum\limits_{2^{-n}\leq l(Q)}|\Delta_n ((f-f_Q)\chi_{\R^d \setminus 2Q})(x)|^2dx\biggl)\lesssim \\ 
&\|f\|_{\BMO}^2 l(Q)^{d-2a}\sum\limits_{2^{-n}\leq l(Q)}2^{n(d-2a)}\lesssim\|f\|_{\BMO}^2,
\end{align*}
and the inequality $\|f\|_D\lesssim\|f\|_{\BMO}$ follows.

Next we proceed to the reciprocal inequality.
\subsection{The reciprocal inequality}
 So, we need to prove the inequality $\|f\|_{\BMO}\lesssim \|f\|_D$. We first notice that for every constant $c$ the following inequality holds:
 \begin{equation}
  \Big(\frac{1}{|Q|}\int_Q |f(x)-f_Q|^2 dx\Big)^{1/2}\lesssim \Big(\frac{1}{|Q|}\int_Q |f(x)-c|^2 dx \Big)^{1/2}.
  \label{first}
 \end{equation}
 We remind the reader that we denote $S_n:=\F^{-1}[\psi_n]$ and $P_n:=S_{n-1}+S_n+S_{n+1}$. 
Hence $\Delta_n f=S_n\ast f$. Let us choose $c := \sum_{2^{-(n-10)}\geq l(Q)} (f\ast S_n)(y)$ where $y$ is the center of $Q$. We infer that
\begin{align}
  &\Big(\frac{1}{|Q|}\int_Q |f(x)-f_Q|^2 dx\Big)^{1/2} \lesssim \nonumber \\
  &\frac{1}{|Q|^{1/2}}\Big(\int_Q\Big|f(x)-\sum_{2^{-(n-10)}\geq l(Q)}\int_{\R^d}f(t)S_n(y-t)dt\Big|^2 dx\Big)^{1/2}. \label{oneandahalf}
\end{align} 
\noindent
Expressing $f(x)$ as $\sum_{n\in\Z}\int_{\R^d}f(t)S_n(x-t)dt$, one readily sees that the right--hand side in (\ref{oneandahalf}) is less than or equal to
\begin{align}
 &\frac{1}{|Q|^{1/2}}\Big(\int_Q\Big|\sum_{2^{-n}\leq l(Q)}\int_{\R^d}f(t)S_n(x-t)dt\Big|^2 dx\Big)^{1/2} + \label{main1}\\
 &\frac{1}{|Q|^{1/2}} \Big(\int_Q \Big| \sum_{2^{-n}\geq l(Q)} \int_{\R^d} f(t)(S_n(x-t)-S_n(y-t))dt \Big|^2 dx \Big)^{1/2}=:
 B_1+B_2.\label{main2}
\end{align}
We estimate $B_1$ and $B_2$ separately. In both estimates, the essential ingredient is a suitable decomposition of 
$\R^d$ into a union of cubes. While estimating $B_2$ we shall also use the cancellation property of the kernels $S_n$ in a similar way 
as it is done in the proof of the H\"ormander---Mihlin multiplier theorem.

\subsubsection{Estimate of term $B_1$}
We need to prove that
\begin{equation}
 \Big(\int_Q\Big|\sum_{2^{-(n-10)}\leq l(Q)}\int_{\R^d}f(t)S_n(x-t)dt\Big|^2 dx\Big)^{1/2}\lesssim |Q|^{1/2} \|f\|_D.
\label{fourandahalph}
\end{equation}
We estimate the square of the left--hand side in (\ref{fourandahalph}) with the help of the fact that $S_n=P_n\ast S_n$:  
\begin{align*}
 &\int_Q \Big|\sum_{2^{-(n-10)}\leq l(Q)}\int_{\R^d}(f \ast P_n)(u) S_n(x-u) du \Big|^2 dx\lesssim\\
 &\int_Q \Big| \sum_{2^{-(n-10)}\leq l(Q)}\int_{\R^d\setminus Q} (f\ast P_n)(u) S_n(x-u) du \Big|^2 dx+\\
 &\int_{\R^d}\Big| \sum_{2^{-(n-10)}\leq l(Q)} \int_Q (f\ast P_n)(u) S_n(x-u)du \Big|^2 dx.
\end{align*}
Hence we are done once we prove the following two inequalities:
\begin{align}
 &B_3:=\int_Q \Big| \sum_{2^{-(n-10)}\leq l(Q)}\int_{\R^d\setminus Q} (f\ast P_n)(u) S_n(x-u) du \Big|^2 dx
 \lesssim |Q| \|f\|_D^2, \label{seventh} \\
 &B_4:=\int_{\R^d}\Big| \sum_{2^{-(n-10)}\leq l(Q)} \int_Q (f\ast P_n)(u) S_n(x-u)du \Big|^2 dx
 \lesssim |Q| \|f\|_D^2. \label{eighth}
\end{align}
Thanks to the triangle inequality in the space $L^2$, the square root of the left--hand side in (\ref{seventh}) is less than or equal to
\begin{equation*}
 \sum_{2^{-(n-10)}\leq l(Q)} \Big( \int_Q \Big( \int_{\R^d \setminus Q} |S_n(x-u)|\cdot|(f\ast P_n)(u)|du \Big)^2 dx \Big)^{1/2}.
\end{equation*}
In order to estimate this expression, we denote summands there by $G_n$ and we are going to estimate
\begin{equation*}
 G_n^2=\Big( \int_Q \Big( \int_{\R^d \setminus Q} |S_n(x-u)|\cdot|f\ast P_n(u)|du \Big)^2 dx.
\end{equation*}
First, we prove the following simple statement:
\begin{lem}
If $\sigma$ is a cube with side length at least $2^{-(n-1)}$, then the following inequality holds:
$$
\bigg(\frac{1}{|\sigma|}\int_\sigma |f\ast P_n|^2\bigg)^{1/2}\lesssim \|f\|_D.
$$
\end{lem}
\begin{proof}
Let us use the definition of $P_n$ and write:
$$
\bigg(\frac{1}{|\sigma|}\int_\sigma |f\ast P_n|^2\bigg)^{1/2}=\bigg(\frac{1}{|\sigma|}\int_\sigma |\Delta_{n-1}f+\Delta_n f+\Delta_{n+1}f|^2\bigg)^{1/2}\lesssim\|f\|_D,
$$
where the last inequality is obvious due to the definition of the norm $\|f\|_D$.
\end{proof}
Denote by $Q^-$ the cube with the same center as $Q$ and with the length of the edge equal to the length of the edge of $Q$ minus $2^{-(n-5)}.$ It is going to be convenient for us to fix a decomposition of the whole space $\R^d$ into a family of cubes $\{\sigma_k\}_{k\in \Z^d}$ with edges of the order $2^{-(n-4)}$ such that $Q^-=\bigcup_{j\in J}\sigma_{k_j}.$ Denote by $G_1$ those cubes from $\{\sigma_k\}$ that are contained in $Q^{-}$, and by $G_2$ those cubes from the family that have empty intersection with the cubes in $G_1$. Note that in this case $\R^d\setminus Q$ is contained in the union of cubes from $G_2$. We infer that
\begin{align*}
&I:=\int_{Q^-}\Big(\int_{\R^d\setminus Q}|S_n(x-u)|\cdot|f\ast P_n(u)|du \Big)^2dx \leq \\
&\sum\limits_{\sigma_l\in G_1}\int_{\sigma_l}\Big(\sum\limits_{\sigma_k\in G_2}\int_{\sigma_k}|S_n(x-u)|\cdot|f\ast P_n(u)|du \Big)^2dx\leq \\
&\sum\limits_{\sigma_l\in G_1}\int_{\sigma_l} \left(\sum\limits_{\sigma_k\in G_2}\Big(\int_{\sigma_k}\frac{|f\ast P_n(u)|^2}{|x-u|^{2a}}du\Big)^{1/2}\Big(\int_{\sigma_k}|S_n(x-u)||x-u|^{2a}du\Big)^{1/2}\right)^2dx\leq \\
&\sum\limits_{\sigma_l\in G_1}\int_{\sigma_l} \Big(\sum\limits_{\sigma_k\in G_2}\int_{\sigma_k}\frac{|f\ast P_n(u)|^2}{|x-u|^{2a}}du\Big)\Big(\sum\limits_{\sigma_k\in G_2}\int_{\sigma_k}|S_n(x-u)||x-u|^{2a}du\Big)dx\lesssim \ldots \;.
\end{align*}
Note that the second sum over $\sigma_k\in G_2$ here is obviously less than or equal to $\int_{\R^d}|x|^{2a}|S_n(x)|^2dx,$ which in turn does not exceed $2^{nd-2na}$, thanks to the Plancherel theorem. We continue the estimate, now referring to the fact that $|x-u|\asymp 2^{-n}|k-l|,$ once $x\in \sigma_k, y\in \sigma_l:$
\begin{multline*}
\ldots\lesssim\sum\limits_{\sigma_l\in G_1}\int_{\sigma_l}\bigg(\sum\limits_{\sigma_k\in G_2}\int_{\sigma_k}|f\ast P_n(u)|^2du\cdot 2^{2na}|k-l|^{-2a}\bigg)\cdot2^{nd-2na}dx\lesssim\ldots\;.
\end{multline*}
Using the lemma, we first write $\int_{\sigma_k} |f\ast P_n|^2 \lesssim 2^{-nd} \|f\|^2_D$ and then continue the chain of the inequalities:
\begin{align*}
 &\ldots\lesssim \sum\limits_{\sigma_l\in G_1}\int_{\sigma_l}\sum\limits_{\sigma_k\in G_2} ||f||_D^2 2^{-nd} 2^{2na} |k-l|^{-2a} 2^{nd-2na}dx \lesssim\\ 
 &||f||_D^2 2^{-nd}\sum\limits_{\sigma_l\in G_1}\sum\limits_{\sigma_k\in G_2}|k-l|^{-2a}.
\end{align*}
As we have already noticed, $|k-l|\asymp 2^n|x-u|$ once $x\in \sigma_k, y\in \sigma_l,$ and hence
\begin{equation*}
2^{-2nd}|k-l|^{-2a}\asymp \int_{\sigma_l}\int_{\sigma_k}|k-l|^{-2a}dudx\asymp 2^{-2an}\int_{\sigma_l}\int_{\sigma_k}|x-u|^{-2a}dudx,
\end{equation*}
and consequently $|k-l|^{-2a}\asymp 2^{2nd}2^{-2an}\int_{\sigma_l}\int_{\sigma_k}|x-u|^{-2a}.$ Let us use this inequality in the estimate of the term $I.$
\begin{align*}
&I\lesssim ||f||_D^2 2^{nd-2an} \sum\limits_{\sigma_l\in G_1}\sum\limits_{\sigma_k\in G_2}\int_{\sigma_l}\int_{\sigma_k} |x-u|^{-2a}dudx=\\
&||f||_D^2 2^{nd-2an} \int_{Q^-}\int_{\R^d \setminus Q'}|x-u|^{-2a}dudx,
\end{align*}
where by $Q'$ we denote the cube which is the union of all cubes that do not lie in $G_2$. In fact, $Q'$ is the cube with the same center as $Q^-$ and with the length of the edge equal to the length of the edge of $Q^-$ plus $2^{-(n-3)}.$

Let us now concentrate on the integral 
$$
\int_{Q^-}\int_{\R^d\setminus Q'}|x-u|^{-2a}dudx.
$$ 
For sake of convenience we shift the cube $Q^-$ in the way that $Q^-=[0,l(Q^-)]^d.$ With no loss of generality we take the inner integral here over the set 
\begin{multline*}
\{u=(u_1,\ldots,u_d)\in \R^d: u_1,\ldots u_k\in[-2^{-(n-3)},l(Q^-)+2^{-(n-3)}],\\ u_{k+1},\ldots,u_d\in (-\infty,-2^{-(n-3)})\},
\end{multline*}
where $0\leq k\leq d-1.$ It is obvious that $|x-u|$ can be substituted by $|x_1-u_1|+\ldots+|x_d-u_d|.$ Direct calculation of the corresponding integral shows that
\begin{align*}
&\int_{-2^{-(n-3)}}^{l(Q)+2^{-(n-3)}} (|x_1-u_1|+\ldots+|x_d-u_d|)^{-2a}du_1 \lesssim \\
&(|x_2-u_2|+\ldots+|x_d-u_d|)^{-2a+1}.
\end{align*}
After that we integrate the previous line $k-1$ times with respect to variables $u_1,\ldots,u_k$ the
right-hand side transforms to $(|x_{k+1}-u_{k+1}|+\ldots+|x_d-u_d|)^{-2a+k}$. The remaining integrals
can be calculated explicitly:
\begin{align*}
&\int_{-\infty}^{-2^{-(n-3)}} (|x_{k+1}-u_{k+1}|+\ldots+|x_d-u_d|)^{-2a+k}du_{k+1} \lesssim \\
&(|x_{k+1}+2^{-(n-3)}|+|x_{k+2}-u_{k+2}|+\ldots+|x_d-u_d|)^{-2a+k+1}.
\end{align*}
After integrating this inequality with respect to the remaining variables, we see that it is left to estimate the integral $$\int_0^{l(Q^-)}\ldots\int_0^{l(Q^-)}(|x_{k+1}+2^{-(n-3)}|+\ldots+|x_{d}+2^{-(n-3)}|)^{-2a+d} dx_1\ldots x_d.$$ It is obvious that this integral is going to be maximal when $k=d-1.$ In the latter case it is less than or equal to 
$$
H:=l(Q^-)^{d-1}\int_0^{l(Q^-)}(x_d+2^{-(n-3)})^{-2a+d}dx_d.
$$ 
In order to bound $H$, we treat two different cases separately. First, if $d$ is even, then $a=(d+2)/2$ and hence $H\leq l(Q)^{d-1}2^n.$ In this case we readily conclude that 
$$
I\lesssim ||f||_D^2 l(Q)^{d-1}2^{-n}.
$$ 
If $d$ is odd, then $d=(d+1)/2$ and 
$$
H\leq l(Q)^{d-1}(\log(l(Q)+2^{-n})-\log(2^{-n}))=l(Q)^{d-1}\log(1+2^nl(Q)).
$$
These inequalities yield that
$$
I\lesssim ||f||_D^2 l(Q)^{d-1}2^{-n}\log(1+2^nl(Q)).
$$ 

This implies that in either case
\begin{align*}
\label{kuminus}
\int_{Q^-}\Big(\int_{\R^d\setminus Q}|f\ast P_n(u)|\cdot|S_n(x-u)|du\Big)^2dx\lesssim ||f||_D^2 l(Q)^{d-1}2^{-n}\log(1+2^nl(Q)).
\end{align*}
In order to finish the estimate of the term $G_n^2$ it is left to bound 
$$
\int_{Q\setminus Q^-}\Big(\int_{\R^d\setminus Q}|f\ast P_n(u)|\cdot|S_n(x-u)|du\Big)^2dx
$$ 
from above. To this end, we chop up $Q\setminus Q^-$ into the cubes $\{q_l \}$ with edges of length of order $2^{-(n-4)}.$ There are about $l(Q)^{d-1}2^{n(d-1)}$ of these. Denote by $q_l^+$ the cube with the same center as $q_l$ and with the length of the edge equal to $2^{-(n-5)}$ plus the length of the edge of $q_l$. In analogy with what we have done above, we obtain the following 
\begin{align*}
&\int_{q_l}\Big(\int_{\R^d\setminus q_l^+}|f\ast P_n(u)|\cdot|S_n(x-u)|du\Big)^2dx\lesssim ||f||_D^2  l(q_l)^{d-1}2^{-n}2\log(1+2^n l(q_l))\lesssim\\ 
&||f||_D^2 2^{-nd}.
\end{align*}

In order to estimate the integral that rest we first use the H\"older inequality, and then the facts that $\int_{q_l^+}|f\ast P_n(u)|^2du\lesssim ||f||_D^2 2^{-nd}$ and  $\int_{q_l^+}|S_n(x-u)|^2du\leq \|S_n\|_{L^2}^2=\|\psi_n\|_{L^2}^2\lesssim 2^{nd}$:
\begin{align*}
&\int_{q_l}\Big(\int_{q_l^+}|f\ast P_n(u)|\cdot|S_n(x-u)|du\Big)^2dx\lesssim \\
&\int_{q_l}\Big(\int_{q_l^+}|f\ast P_n(u)|^2du\cdot\int_{q_l^+}|S_n(x-u)|^2du\Big)dx\lesssim \|f\|_D^2 2^{-nd}.
\end{align*}
Taking into account the fact that there are about $l(Q)^{d-1}2^{n(d-1)}$ of cubes in the collection $\{q_l\}$ we infer that
\begin{align*}
&\int_{Q\setminus Q'}\Big(\int_{\R^d\setminus Q}|f\ast P_n(u)|\cdot|S_n(x-u)|du\Big)^2dx\lesssim \\
&\int_{q_l}\Big(\int_{q_l^+}|f\ast P_n(u)|^2du\cdot\int_{q_l^+}|S_n(x-u)|^2du\Big)dx\lesssim \|f\|_D^2 2^{-nd}\lesssim \|f\|_D^2 l(Q)^{d-1} 2^{-n}.
\end{align*}
 Finally the estimates yield
$$G_n^2\lesssim \|f\|_D^2 l(Q)^{d-1} 2^{-n} \log(1+ 2^nl(Q)).$$
Now we have all the ingredients to finish the estimate of the term $B_3$:
\begin{align*}
&\frac{B_3^{1/2}}{|Q|^{1/2}}\lesssim \frac{1}{|Q|^{1/2}}\sum\limits_{n\geq -\log_2l(Q)}\|f\|_D l(Q)^{(d-1)/2} 2^{-n/2}\log_2(1+2^nl(Q))^{1/2}\asymp\\
&\|f\|_D l(Q)^{-1/2}\sum\limits_{m\geq 0}2^{-(m-\log_2l(Q))/2}\log_2(1+2^{m-\log_2l(Q)}l(Q))^{1/2}=\\
&\|f\|_D \sum\limits_{m\geq 0}2^{-m/2}\log_2(1+2^m)^{1/2} \lesssim \|f\|_D,
\end{align*}
and $B_3$ is estimated. 

We proceed to the inequality (\ref{eighth}). Note that the functions
\begin{align*}
 &x\mapsto \int_Q (f\ast P_n)(u) S_{n_1}(x-u) du,\\
 &x\mapsto \int_Q (f\ast P_n)(u) S_{n_1}(x-u) du
\end{align*}
are orthogonal in $L^2(\R^d)$ once we suppose that $|n_1-n_2|\geq 2.$
This is a direct consequence of the fact that the functions $\psi_{n_1}$ and $\psi_{n_2}$ have non--intersecting supports. Thanks to this orthogonality we can now estimate the expression $B_4$ from the left--hand side of (\ref{eighth}):
\begin{align*}
 &B_4=\sum_{2^{-(n-10)}\leq l(Q)} \int_{\R^d} \Big| \int_Q (f\ast P_n)(u) S_n(x-u) du \Big|^2 dx=\\
 &\sum_{2^{-(n-10)}\leq l(Q)} \int_{\R^d \setminus Q} \Big| \int_Q (f\ast P_n)(u) S_n(x-u) du \Big|^2 dx+\\
 &\sum_{2^{-(n-10)}\leq l(Q)} \int_{Q} \Big| \int_Q (f\ast P_n)(u) S_n(x-u) du \Big|^2 dx.
\end{align*}
The estimate of the first term here can be performed in the same way as the estimate of the term $B_3.$ The second term is less than or equal to two times the sum
\begin{align*}
 &\sum_{2^{-(n-10)}\leq l(Q)} \int_Q \Big|  \int_{\R^d} (f\ast P_n)(u) S_n(x-u) du \Big|^2 dx+\\
 &\sum_{2^{-(n-10)}\leq l(Q)} \int_Q \Big|  \int_{\R^d\setminus Q} (f\ast P_n)(u) S_n(x-u)  du \Big|^2 dx.
\end{align*}
The second sum here has been already estimated. Taking into account the fact that $P_n\ast S_n=S_n$, we infer that the first one equals
\begin{equation*}
 \sum_{2^{-(n-10)}\leq l(Q)}\int_Q \Big| \int_{\R^d} f(t)S_n(x-t)dt \Big|^2 dx,
\end{equation*}
which in turn does not exceed $\|f\|_D^2 |Q|$ by the definition of the norm $\|f\|_D$. Hence the term $B_1$ is estimated.

\subsubsection{Estimate of the term $B_2$}
In order to finish the proof of the theorem we need to prove the following inequality:
$$
B_2=\frac{1}{|Q|^{1/2}} \Big( \int_Q \Big| \sum_{2^{-(n-10)}>l(Q)} \int_{\R^d} f(t)(S_n(x-t)-S_n(y-t)) dt \Big|^2 dx \Big)^{1/2}\lesssim \|f\|_D.
$$
We remind the reader that $y$ here is the center of the (fixed) cube $Q$. Once again using the fact that $S_n=S_n\ast P_n$, we infer that the following inequality holds:
$$
B_2\leq \frac{1}{|Q|^{1/2}} \Big( \int_Q \Big| \sum_{2^{-(n-10)}>l(Q)} \int_{\R^d} |(f\ast P_n)(u)| |S_n(x-u)-S_n(y-u)| du \Big|^2 dx \Big)^{1/2}.
$$
Without loss of generality, we assume that $Q$ is centered at $0$ so that $y=0$. We are now going to estimate the terms of this sum. In more details we shall consider the following expression:
\begin{equation}
\label{toxamain}
\int_{\R^d} |(f\ast P_n)(-u)| |S_n(u-x)-S_n(u)|du.
\end{equation}

Let us decompose $\R^d$ into the cubes $\{\delta_k\}_{k\in\Z^d}$ with edge length $2^{-(n-10)}$ (so that the center of $\delta_k$ is
$2^{-(n-10)}k$). We denote by $G_m$, $m\in\mathbb{N}$, the union of the cubes from $\{\delta_k\}$ such that the maximal coordinates of their centers lie 
between $2^{-(n-10)+m}$ and $2^{-(n-10)+m+1}$ and by $G_0$ the union of the remaining cubes. 
Hence we can write the following inequality:

\begin{align}
&\sum_{m=1}^\infty \int_{G_m} |(f\ast P_n)(-u)| |S_n(u-x)-S_n(u)|du\leq \nonumber \\
&\sum_{m=1}^\infty \Big( \int_{G_m} |(f\ast P_n)(u)|^2 du \Big)^{1/2} \Big( \int_{G_m}|S_n(u-x)-S_n(u)|^2 du \Big)^{1/2}. \label{toxa1}
\end{align}
Lemma 2 yields
$$
\Big( \int_{G_m} |(f\ast P_n)(u)|^2 du \Big)^{1/2}\lesssim |G_m|^{1/2}\|f\|_D\asymp 2^{md/2}\cdot 2^{-nd/2}\cdot \|f\|_D.
$$

Let us now concentrate on the second factor in~\eqref{toxa1}:
$$
\Big( \int_{G_m}|S_n(u-x)-S_n(u)|^2 du \Big)^{1/2} \leq \Big( \int_{\{|u|\geq 2^{-n+m}\}}|S_n(u-x)-S_n(u)|^2 du \Big)^{1/2}.
$$ 

In order to estimate this integral, the cancellation of $\widetilde{S_n}(u):=S_n(u-x)-S_n(u)$ must come into play. For
any $r>1$ we use the notation $r'=r/(r-1)$. Let us fix a number $q$, $1<q<\infty$, which is greater than 
$d/(2a-d)$. The H\" older and the Hausdorff--Young inequalities yield
\begin{align*}
&\int_{|u|\geq 2^{-n+m}} |\widetilde{S_n}(u)|^2 du \lesssim 
\sum_{j=1}^d \Big( \int_{\R^d} |u_j^a \widetilde{S_n}(u)|^{2q'} \Big)^{1/q'} \Big( \int_{\{|u|\geq 2^{-n+m}\}} |u|^{-2aq} du \Big)^{1/q}\asymp\\
&\sum_{j=1}^d 2^{(-n+m)(-2aq+d)/q} \|u_j^a \widetilde{S_n}(u)\|_{2q'}^2\lesssim  \sum_{j=1}^d 2^{(-n+m)(-2aq+d)/q} 
\|D_{\xi_j}^a \widehat{\widetilde{S_n}}(\xi)\|_{(2q')'}^2.
\end{align*}
Note that $(2q')'=2q/(q+1)$. Hence the expression~\eqref{toxa1} does not exceed
\begin{align*}
&\sum_{m=1}^\infty 2^{md/2}\cdot 2^{-nd/2}\cdot\|f\|_D\cdot 2^{(-n+m)(-2aq+d)/2q}\cdot
\max_j \|D_{\xi_j}^a \wh{\wt{S_n}}(\xi)\|_{\frac{2q}{q+1}}=\\
&\|f\|_D\cdot\max_j  \|D_{\xi_j}^a \wh{\wt{S_n}}(\xi)\|_{\frac{2q}{q+1}} \cdot 2^{-nd/2}\cdot 2^{na} \cdot 2^{-nd/2q}
\cdot\sum_{m=1}^\infty 2^{m(\frac{d}{2}-a+\frac{d}{2q})}.
\end{align*}
Since $q$ is greater than $d/(2a-d)$, the sum here is equal to some finite constant and we are left with the expression
\begin{equation*}
\|f\|_D\cdot\max_j  \|D_{\xi_j}^a \wh{\wt{S_n}}(\xi)\|_{\frac{2q}{q+1}} \cdot 2^{-nd/2}\cdot 2^{na} \cdot 2^{-nd/2q}.
\end{equation*}
In order to estimate $ \|D_{\xi_j}^a \wh{\wt{S_n}}(\xi)\|_{2q/(q+1)}$ we note that 
$\wh{\wt{S_n}}(\xi)=\psi_n(\xi)(e^{2\pi i x\cdot\xi}-1)$ and hence $\supp \wh{\wt{S_n}}\subset\{2^{n-1}\leq|\xi|<2^{n+1}\}$. It is easy to see that if $\gamma\leq a$, then for $x\in Q$, $\xi\in\supp \wh{\wt{S_n}}$ there holds
$|D_{\xi_j}^a[e^{2\pi i x\cdot\xi}-1]|\lesssim l(Q)2^{n(1-\gamma)}$. Indeed, for $\gamma=0$ it is true because 
$|e^{2\pi i x\cdot\xi}-1|\lesssim |x\cdot\xi|\lesssim 2^n l(Q)$ and for $\gamma>0$ it is a consequence of the inequality
$D_{\xi_j}^\gamma[e^{2\pi i x\cdot \xi}]\lesssim |x|^\gamma\lesssim l(Q)(2^{-n})^{\gamma-1}$.

Using this estimate and the Leibnitz rule we infer that
$$
|D_{\xi_j}^a \wh{\wt{S_n}}(\xi)|\lesssim l(Q)\sum_{r=0}^a |D_\xi^r \psi_n(\xi)|\cdot 2^{-n(a-r-1)}.
$$
This allows us to write
$$
\|D_{\xi_j}^a \wh{\wt{S_n}}(\xi)\|_{\frac{2q}{q+1}} \lesssim l(Q) \cdot\max_{0\leq r\leq a} 2^{-n(a-r-1)}\|D_{\xi_j}^r\psi_n(\xi)\|_{\frac{2q}{q+1}}.
$$
The norm here can be estimated by means of the H\" older inequality:
\begin{align*}
&\|D_{\xi_j}^r\psi_n(\xi)\|_{\frac{2q}{q+1}}^{\frac{2q}{q+1}}= \int_{2^{n-1}\leq |\xi|< 2^{n+1}} |D_{\xi_j}^r \psi_n(\xi)|^{\frac{2q}{q+1}} d\xi \leq\\ 
&|{\{2^{n-1}\leq |\xi|< 2^{n+1}\}}|^{1/(q+1)} \Big(\int |D_{\xi_j}^r \psi_n(\xi)|^2 d\xi \Big)^{q/(q+1)}.
\end{align*}
The first factor in the last line is less than or equal to $2^{nd/(q+1)}$ whereas the second one does not exceed $(2^{nd-2nr})^{q/(q+1)}$ thanks to the third assumption of the theorem. So we conclude that the following inequality holds:
$$
\|D_{\xi_j}^r\psi_n(\xi)\|_{\frac{2q}{q+1}} \lesssim 2^{\frac{nd}{2q}} \cdot 2^{\frac{nd}{2}-nr}.
$$
Finally, collecting all the estimates, we see that the expression~\eqref{toxa1} is bounded from above by $2^n l(Q) \|f\|_D$.

In order to finish the estimate of the expression~\eqref{toxamain} it is left to deduce a similar bound for the integral
$$
\int_{G_0} |f\ast P_n(-u)| |S_n(u-x)-S_n(u)| du.
$$
Note that $G_0$ is composed of the fixed number of cubes: this number depends on the dimension $d$ only. Hence we are done once we estimate integrals over $\delta$ for $\delta\in G_0$, namely
$$
\int_{\delta} |f\ast P_n(-u)| |S_n(u-x)-S_n(u)| du.
$$
Using basic properties of the Fourier transformation along with the H\" older inequality, we infer that
\begin{align*}
&|S_n(u-x)-S_n(u)|\lesssim l(Q) \|\nabla S_n\|_{\infty}\lesssim l(Q) \||\xi|\psi_n(\xi)\|_1\lesssim\\ 
&l(Q)2^n\|\psi_n\|_{1} \lesssim l(Q)2^n 2^{nd/2}\|\psi_n\|_2\lesssim 2^{nd} l(Q) 2^n.
\end{align*}
This implies the following bound
\begin{align*}
&\int_{\delta} |f\ast P_n(-u)| |S_n(u-x)-S_n(u)| du\lesssim 2^{nd} l(Q) 2^n\int_\delta |f\ast P_n(-u)| du\leq\\
&2^{nd} l(Q) 2^n |\delta|^{1/2} \bigg(\int_\delta |f\ast P_n(-u)|^2 du\bigg)^{1/2}\lesssim 2^{nd} |\delta| l(Q) 2^n \|f\|_D\lesssim\|f\|_D l(Q) 2^n.
\end{align*}
Note that here we once again used Lemma 2.

Finally, we are ready to finish off the estimate of term $B_2$ and thus the proof of the theorem:
$$
B_2\lesssim \sum_{2^{-(n-10)}>l(Q)} \|f\|_D\cdot l(Q)\cdot 2^n \lesssim \|f\|_D.
$$

\end{proof}

\end{document}